\documentclass[a4paper,twoside,11pt,reqno]{amsart}
\usepackage[margin=1in]{geometry}
\usepackage[english]{babel}
\usepackage[utf8]{inputenc}
\usepackage{amsmath}
\usepackage{url} 
\usepackage{amssymb}
\usepackage{amsfonts}
\usepackage{amsthm}
\usepackage{bm}
\usepackage{mathrsfs}
\usepackage{hyperref}
\usepackage{comment}
\hypersetup{
	colorlinks=true,
	linkcolor=blue,
	filecolor=blue,
	citecolor=purple,      
	urlcolor=cyan,
}
\usepackage{xcolor}
\usepackage{dsfont}
\usepackage{physics}
\usepackage{stmaryrd}
\usepackage{enumitem}

\usepackage{tikz-cd}
\usepackage{tikz}
\usepackage{tikzsymbols}
\usepackage{mathtools}
\usetikzlibrary{matrix}

\usepackage[matrix,arrow,cmtip]{xy} 
\SelectTips{cm}{}
\newdir{(}{{}*!/-5pt/@^{(}}
\newdir{(x}{{}*!/-5pt/@_{(}}
\newdir{+}{{}*!/-9pt/{}}
\newdir{>+}{@{>}*!/-9pt/{}}
\entrymodifiers={+!!<0pt,\fontdimen22\textfont2>}

\theoremstyle{plain}
\newtheorem{theorem}[subsubsection]{Theorem}
\newtheorem{conjecture}[subsubsection]{Conjecture}

\newtheorem{lemma}[subsubsection]{Lemma}
\newtheorem{proposition}[subsubsection]{Proposition}
\newtheorem{corollary}[subsubsection]{Corollary}
\theoremstyle{definition}

\newtheorem{acknowledgement}[subsubsection]{Acknowledgement}

\newtheorem{example}[subsubsection]{Example}
\theoremstyle{remark}
\newtheorem{remark}[subsubsection]{Remark}

\numberwithin{equation}{section}

\newcommand{\OO}{\mathcal{O}}
\newcommand{\HH}{\mathrm{H}}

\newcommand{\Br}{\mathrm{Br}}

\newcommand{\Coh}{\mathrm{Coh}}

\newcommand{\cA}{\mathcal{A}}

\begin{document}
\title{Picard group action on the category of twisted sheaves}
		\date{}
\author{Ting Gong, Yeqin Liu, and Yu Shen}

\address{Department of Mathematics,
University of Washington, C-138 Padelford, Seattle, WA 98195, USA}
\email{tgong2@uw.edu}
  \address{Department of Mathematics, University of Michigan, 530 Church St,
Ann Arbor, MI 48109, USA}
\email{yqnl@umich.edu}
\address{Department of Mathematics, Michigan State University, 619 Red Cedar Road, East Lansing, MI 48824, USA}
  \email{shenyu5@msu.edu}


\begin{abstract}
In this paper, we study the  category of twisted sheaves over a scheme $X$. Let $\mathcal{M}$ be a quasi-coherent sheaf on $X$, and $\alpha$ in $\Br(X)$. We show that the functor
$
- \otimes_{\mathcal{O}_X} \mathcal{M} : \operatorname{QCoh}(X, \alpha) \to \operatorname{QCoh}(X, \alpha)
$
is naturally isomorphic to the identity functor if and only if $\mathcal{M}\cong \OO_{X}$. As a corollary, the action of $\operatorname{Pic}(X)$ on $D^{b}(X, \alpha)$ is faithful for any Noetherian scheme $X$.

    \end{abstract}
  
\maketitle

\section{Introduction}

For a scheme $X$, let $\Br(X)$ denote the group of Brauer equivalence classes of Azumaya algebras over $X$.
Given $\alpha \in \Br(X)$, let $\operatorname{Coh}(X)$ (resp. $\operatorname{QCoh}(X)$) denote the abelian category of coherent (resp. quasi-coherent) sheaves, and let $\operatorname{Coh}(X, \alpha)$ (resp. $\operatorname{QCoh}(X, \alpha)$) denote the abelian category of $\alpha$-twisted coherent (resp. quasi-coherent) sheaves. 
The following is our main theorem.

\begin{theorem}\label{main theorem module}
We have the following results:
\begin{enumerate}
\item (Theorem \ref{injective of module}) Let $X$ be any scheme and $\mathcal{M} \in \operatorname{QCoh}(X)$. Then the functor
\[
- \otimes_{\mathcal{O}_{X}} \mathcal{M} : \operatorname{QCoh}(X, \alpha) \to \operatorname{QCoh}(X, \alpha)
\]
is naturally isomorphic to the identity functor $\operatorname{id}$ if and only if $\mathcal{M} \cong \mathcal{O}_{X}$.

\item (Theorem \ref{injective of Coh}) Assume further that $X$ is Noetherian and $\mathcal{M} \in \operatorname{Coh}(X)$. Then the functor
\[
- \otimes_{\mathcal{O}_{X}} \mathcal{M} : \operatorname{Coh}(X, \alpha) \to \operatorname{Coh}(X, \alpha)
\]
is naturally isomorphic to the identity functor $\operatorname{id}$ if and only if $\mathcal{M} \cong \mathcal{O}_{X}$.
\end{enumerate}
\end{theorem}

Note that Theorem \ref{main theorem module} (i) \textbf{hold for any scheme} $X$ and any $\alpha \in \Br(X)$.
We will prove this result using non‑commutative ring theory.

Theorem~\ref{main theorem module} leads to many interesting and important consequences. For example,   let $D^{b}(X,\alpha) := D^{b}(\Coh(X, \alpha))$ denote the bounded derived category of $\operatorname{Coh}(X,\alpha)$.
 Then the Picard group $\operatorname{Pic}(X)$ has a natural action on $D^{b}(X,\alpha)$ via tensoring with a line bundle. As an application of Theorem \ref{main theorem module}, we know that this action is faithful when $X$ is Noetherian.

\begin{corollary}[{Theorem \ref{injective of Coh}}]\label{main theorem derived categiory}

Assume that $X$ is a Noetherian scheme.    Then the map $\Psi$ is injective, where $\Psi$ is defined as follows:
    \[
    \Psi: \operatorname{Pic}(X) \to \operatorname{Aut}(D^{b}(X, \alpha)), \quad \mathcal{L} \mapsto \Psi(\mathcal{L}) : \mathcal{F}^{\bullet} \mapsto \mathcal{F}^{\bullet} \otimes_{\mathcal{O}_{X}} \mathcal{L}.
    \]       
\end{corollary}
\begin{remark}\label{remark fourier}
In \cite[Theorem~1.1]{canonaco2007twisted}, Canonaco and Stellari show that for a smooth projective variety $X$ over a field $k$, if the Fourier--Mukai transform
\[
\Phi_{\mathcal{P}} \colon D^{b}(X, \alpha) \longrightarrow D^{b}(X, \alpha)
\]
is naturally isomorphic to the identity functor, then $\mathcal{P} \cong \Delta_{*}\mathcal{O}_{X}$, where $\Delta \colon X \to X \times X$ is the diagonal embedding. Therefore, Corollary~\ref{main theorem derived categiory} holds when $X$ is smooth and projective over $k$. However, it is not clear whether their method can be generalized to arbitrary schemes.
\end{remark}

Now assume that  $X$ is a smooth projective variety over a field $k$. Let $\mathscr{X}\to X$ be the $\mathbb{G}_{m}$-gerbe associated to the Brauer class $\alpha$. Let $\mathscr{A}ut_{\mathscr{X}}$ be the fibered category which to any $k$-scheme $S$ associates the groupoid of isomorphism $\mathscr{X}\to \mathscr{X}$ inducing the identity on the stabilizer group schemes $\mathbb{G}_{m}$. Then $\mathscr{A}ut_{\mathscr{X}}$ acts on $D^{b}(X,\alpha)$. As a Corollary of Theorem \ref{main theorem module}, we show that this action is faithful, which answers a question of Olsson in \cite[Chapter 6.4]{olsson2025twisted}.
\begin{corollary}[{Corollary \ref{action faithful}}]
   Assume that $k$ is algebraically closed. Let $\sigma\in \mathscr{A}ut_{\mathscr{X}}(k)$ be an automorphism of the stack inducing the identity functor on $D^{b}(X,\alpha)$. Then we have $\sigma=\operatorname{id}$.
\end{corollary}

It follows from our results that taking Brauer twists of varieties does not yield new Calabi–Yau categories. More precisely, we show that $D^{b}(X)$ is Calabi–Yau if and only if $D^{b}(X, \alpha)$ is Calabi–Yau.

\begin{theorem}[{Theorem \ref{Calabi-Yau}}]\label{main theorem calabi-yau}
    Assume $\operatorname{dim}X=n$. Then $D^{b}(X)$ is a Calabi-Yau $[n]$-category if and only if $D^{b}(X,\alpha)$ is a Calabi-Yau $[n]$-category.
\end{theorem}

In \cite[Example~6.6]{olsson2025twisted}, Olsson gives an example showing that Corollary~\ref{main theorem derived categiory} does not hold for Deligne--Mumford stacks. We provide a new example (see Example~\ref{example}) demonstrating that all of the theorems above—Theorems~\ref{main theorem module}, \ref{main theorem derived categiory}, and~\ref{main theorem calabi-yau}—fail for Deligne--Mumford stacks. However, the examples in \cite{olsson2025twisted} and this paper are both certain gerbes over varieties. It would be interesting to know if such examples exist for Deligne--Mumford stacks with generically trivial stabilizers.

\begin{conjecture}[{Conjecture~\ref{conjecture}}]
Let $X$ be a quasi-separated Deligne--Mumford stack with generically trivial stabilizer. 
Let $\mathcal{M}\in \operatorname{QCoh}(X)$ and $\alpha\in \Br(X)$.
Then the functor
\[
-\otimes_{\mathcal{O}_{X}}\mathcal{M} \colon \operatorname{QCoh}(X,\alpha)\longrightarrow \operatorname{QCoh}(X,\alpha)
\]
is naturally isomorphic to the identity functor $\operatorname{id}$ if and only if $\mathcal{M}\cong \mathcal{O}_{X}$.
\end{conjecture}
\begin{acknowledgement}
We are grateful to James Hotchkiss, Alexander Kuznetsov, and Paolo Stellari for carefully reading early drafts and providing many useful comments. We would like to thank Martin Olsson for pointing out Remark~\ref{remark fourier}. We also thank Rajesh Kulkarni and Alexander Perry for many helpful discussions and comments. TG is particularly grateful to his advisor, Max Lieblich, for his support.
\end{acknowledgement}

\section{Proof of main results}

Let $X$ be a scheme and assume $\alpha\in \Br(X)$. Let $\mathcal{A}$ be an Azumaya algebra over $X$ with $[\mathcal{A}]=\alpha \in \Br(X)$. Let $\operatorname{QCoh}(X,\cA)$ be the abelian category of quasi-coherent right $\cA$-modules.
\begin{lemma}[{\cite[Theorem 1.3.7]{caldararu2000derived}}]\label{Cada}
 We have an equivalence of categories
$$i: \operatorname{QCoh}(X,\alpha)\cong \operatorname{QCoh}(X,\cA).$$
Moreover, let $\mathcal{M}\in \operatorname{QCoh}(X)$ be a quasi-coherent sheaf on $X$, we have the following commutative diagram
\begin{center}
\begin{tikzcd}
\operatorname{QCoh}(X,\alpha) \arrow[r, "- \otimes_{\mathcal{O}_{X}} \mathcal{M}"] \arrow[d, "i"] & \operatorname{QCoh}(X,\alpha) \arrow[d, "i"] \\
\operatorname{QCoh}(X,\cA) \arrow[r, "- \otimes_{\mathcal{O}_{X}} \mathcal{M}"]                & \operatorname{QCoh}(X,\cA).               
\end{tikzcd}
\end{center}
\end{lemma}
By Lemma~\ref{Cada}, from now on we work in the category $\operatorname{QCoh}(X, \mathcal{A})$. 
To prove Theorem~\ref{main theorem module}, we need some algebraic lemmas. 
In what follows, we assume that $R$ is a commutative ring and that $A$ is an Azumaya algebra over $R$. 
Let $M$ be an $R$-module. Then $A \otimes_{R} M$ naturally inherits an $A$-bimodule structure, induced by the standard $A$-bimodule structure on $A$.

\begin{lemma}\label{center of module}
 Let $M$ be an $R$-module, and define
\[
Z(A \otimes_R M) := \{ b \in A \otimes_R M \mid ab = ba \text{ for all } a \in A \}.
\]
Then $Z(A \otimes_R M) \cong M$, via the inclusion $m \mapsto 1 \otimes m$.
\end{lemma}
\begin{proof}
We may assume $X = \operatorname{Spec}(R)$ is connected, and that $\operatorname{rank}_R(A) = n^2$ for some $n \in \mathbb{N}$. Note that we have the following short exact sequence of $R$-modules, where each term is a projective $R$-module:
\[
0 \to R \to A \to A/R \to 0.
\]
Tensoring with $M$, we obtain the short exact sequence
$$
0 \to M \xrightarrow{} A \otimes_R M \to (A/R) \otimes_R M \to 0.
$$
Therefore, we have $M \subseteq Z(A \otimes_R M)$, which induces the following short exact sequence:
\begin{equation} \label{short exact sequence}
0 \to M \xrightarrow{} Z(A \otimes_R M) \to Z(A \otimes_R M)/M \to 0.
\end{equation}

By \cite[Theorem 3.1.1]{colliot2021brauer}, there exists a finite surjective étale morphism $\operatorname{Spec}S\to \operatorname{Spec} R$ such that $A_{S}:=A \otimes_R S \cong \operatorname{Mat}_n(S)$. 
Let $M_S := M \otimes_R S$. Then we have
\[
Z(A_S \otimes_S M_S) = Z(\operatorname{Mat}_{n}(M_S)) = M_S.
\]
Note that there is a natural inclusion
$
Z(A \otimes_R M)_S \subseteq Z(A_S \otimes_S M_S).
$ By pulling back the short exact sequence~\eqref{short exact sequence} to $S$, we obtain the following exact sequence:
\[
0 \to M_S \xrightarrow{} Z(A \otimes_R M)_S \to \left(Z(A \otimes_R M)/M\right)_S \to 0.
\]
 Therefore, we have
\[
 Z(A_S \otimes_S M_S)=M_{S} \subseteq Z(A \otimes_R M)_S \subseteq Z(A_{S}\otimes_{S}M_{S}),
\]
so all three terms are equal. In particular,
\[
Z(A \otimes_R M)_S = M_S \quad \text{and} \quad \left(Z(A \otimes_R M)/M\right)_S = 0.
\]
Since $S$ is faithfully flat over $R$, it follows that $Z(A \otimes_R M)/M = 0$, and hence
\[
Z(A \otimes_R M) = M.
\]
\end{proof}

Let $A^{\operatorname{op}}$ denote the opposite algebra of $A$, and set $A^{e} := A \otimes_{R} A^{\operatorname{op}}$. Note that
\[
\operatorname{Hom}_{A\text{-bimod}}(A, A \otimes_{R} M) = \operatorname{Hom}_{A^{e}}(A, A \otimes_{R} M),
\]
where both $A$ and $A \otimes_{R} M$ are regarded as left $A^{e}$-modules via the standard structure induced by the $A$-bimodule actions.

\begin{lemma}\label{bimodule homorphism}
There is a natural isomorphism of $R$-modules
\[
\eta \colon M \xrightarrow{\;\sim\;} \operatorname{Hom}_{A^{e}}(A, A \otimes_R M), 
\qquad 
m \longmapsto \bigl(a \longmapsto a \otimes m\bigr).
\]
\end{lemma}

\begin{proof}

Since $1 \otimes m \in Z(A \otimes_R M)$, the map $\eta(m)$ is $A$-bilinear; hence
$\eta(m) \in \operatorname{Hom}_{A^{e}}(A, A \otimes_R M)$. Thus $\eta$ is well-defined.

If $m \neq m'$, then
\[
\eta(m)(1)=1 \otimes m \neq 1 \otimes m'=\eta(m')(1),
\]
so $\eta$ is injective.

For surjectivity, let $\psi \in \operatorname{Hom}_{A^{e}}(A, A \otimes_R M)$ and set $u := \psi(1)$. For any $a \in A$,
\[
\psi(a) = \psi\bigl((a \otimes 1)\cdot 1\bigr) = (a \otimes 1)\psi(1) = a\,u,
\qquad
\psi(a) = \psi\bigl((1 \otimes a)\cdot 1\bigr) = (1 \otimes a)\psi(1) = u\,a,
\]
hence $a u = u a$ for all $a \in A$. By Lemma~\ref{center of module}, we conclude $u \in Z(A \otimes_R M) = M$, and therefore
$\psi = \eta(u)$. This shows $\eta$ is surjective.

This completes the proof.
\end{proof}

\begin{remark}
    By   Lemma \ref{bimodule homorphism}, we know for any $\psi\in\operatorname{Hom}_{A^{e}}(A, A \otimes_R M)$, $\psi=\operatorname{id}_{A}\otimes \phi, $ where $\phi\in \operatorname{Hom}_{R}(R,M)$.
\end{remark}

We are now ready to begin the proof of the main theorem. Let $\mathcal{M}\in \operatorname{QCoh}$(X). We first observe the following proposition.

\begin{proposition}\label{Proposition of fix category}
 The functor
\[
- \otimes_{\mathcal{O}_{X}} \mathcal{M} : \operatorname{QCoh}(X, \cA) \to \operatorname{QCoh}(X, \cA)
\]
is naturally isomorphic to the identity functor $\operatorname{id}$ if and only if we have the $\cA$-bimodule isomorphism $\cA \cong \cA\otimes_{\mathcal{O}_{X}} \mathcal{M}$.
\end{proposition}
\begin{proof}

$\Longleftarrow$: If we have $\cA\cong \cA\otimes_{\mathcal{O}_{X}} \mathcal{M}$ as $\cA$-bimodules, then for any  $\mathcal{F}\in \operatorname{QCoh}(X,\cA)$, we have $$\mathcal{F}\cong \mathcal{F}\otimes_{\cA}\cA\cong \mathcal{F}\otimes_{\cA}(\cA\otimes_{\mathcal{O}_{X}}\mathcal{M})\cong \mathcal{F}\otimes_{\mathcal{O}_{X}} \mathcal{M}.$$

$\Longrightarrow$: Suppose that $-\otimes_{\mathcal{O}_{X}}\mathcal{M} \cong \operatorname{id}$ as functors. 
Then $\mathcal{A} \cong \mathcal{A} \otimes_{\mathcal{O}_{X}} \mathcal{M}$ as right $\mathcal{A}$-modules. 
Let $\psi \colon \mathcal{A} \xrightarrow{\sim} \mathcal{A} \otimes_{\mathcal{O}_{X}} \mathcal{M}$ be this isomorphism of right $\mathcal{A}$-modules. 
We want to show that $\psi$ is an $\mathcal{A}$-bimodule isomorphism. For any open subset $U \subseteq X$, we have the map
\[
\psi|_{U} \colon \mathcal{A}|_{U} \longrightarrow \bigl(\mathcal{A}\otimes_{\mathcal{O}_{X}}\mathcal{M}\bigr)|_{U}.
\]
 Since $\psi$ is right $\mathcal{A}$-linear, we have $\psi|_{U}(a)=\psi|_{U}(1)\,a$ for all $a\in \mathcal{A}(U)$. 
To show that $\psi$ is an $\mathcal{A}$-bimodule morphism, it suffices to prove that
\[
\psi|_{U}(1)\,a \;=\; a\,\psi|_U(1)\qquad\text{for all } a\in \mathcal{A}(U).
\]

Let $i \colon U \hookrightarrow X$ be the open immersion. For any $a \in \mathcal{A}(U)$, the left multiplication by $a$
\[
L_a \colon i^{*}\mathcal{A} = \mathcal{A}|_{U} \longrightarrow \mathcal{A}|_{U}, \qquad a_0 \longmapsto a\,a_0,
\]
is a homomorphism of right $i^{*}\mathcal{A}$-modules. It induces a homomorphism of right $\mathcal{A}$-modules
\[
i_{*}L_{a} \colon i_{*}i^{*}\mathcal{A} \longrightarrow i_{*}i^{*}\mathcal{A}.
\]
Thus we have a commutative diagram:
\[
\begin{tikzcd}
i_{*}i^{*}\mathcal{A} \arrow[r,"i_{*}L_{a}"] \arrow[d,"\psi"'] &
i_{*}i^{*}\mathcal{A} \arrow[d,"\psi"] \\
i_{*}i^{*}\mathcal{A}\otimes_{\mathcal{O}_{X}}\mathcal{M}
\arrow[r,"\,i_{*}L_{a}\otimes \operatorname{id}\,"] &
i_{*}i^{*}\mathcal{A}\otimes_{\mathcal{O}_{X}}\mathcal{M}.
\end{tikzcd}
\]
Applying $i^{*}$ and using $i^{*}i_{*} \cong \operatorname{id}$, we obtain the following commutative diagram:
\[
\begin{tikzcd}
i^{*}\mathcal{A} \arrow[r,"L_{a}"] \arrow[d,"\psi|_{U}"'] &
i^{*}\mathcal{A} \arrow[d,"\psi|_{U}"] \\
i^{*}(\mathcal{A}\otimes_{\mathcal{O}_{X}}\mathcal{M})
\arrow[r,"\,L_{a}\otimes \operatorname{id}\,"] &
i^{*}(\mathcal{A}\otimes_{\mathcal{O}_{X}}\mathcal{M}).
\end{tikzcd}
\]
For any $a\in \cA(U)$, we have 
\[
\psi|_{U} \circ L_a(1) = \psi|_{U}(a) = \psi|_{U}(1) a , \quad \text{and} \quad (L_a \otimes \operatorname{id}) \circ \psi|_{U}(1)  = a \psi|_{U}(1).
\]
So we have $\psi|_{U}(1)a =  a\psi|_{U}(1)$ for all $a \in \cA(U)$, and therefore $\psi$ is an $A$-bimodule isomorphism.  
\end{proof}

\begin{remark}
In Proposition~\ref{Proposition of fix category}, the ``only if'' direction requires that $X$ be a scheme. 
We cannot combine Proposition~\ref{Proposition of fix category} with descent to treat Deligne--Mumford stacks. 
The reason is that for an étale morphism $p \colon U \to X$, there is in general no natural isomorphism 
$p^{*}p_{*} \cong \operatorname{id}$ on quasi-coherent sheaves.
\end{remark}

We want to determine the $\mathcal{A}$-bimodule homomorphisms from $\mathcal{A}$ to $\mathcal{A} \otimes_{\mathcal{O}_X} \mathcal{M}$. This will be accomplished by the following proposition.

\begin{proposition}\label{Bimodule of sheaf}We have the isomorphism
$$ \HH^{0}(X,\mathcal{M})\cong \operatorname{Hom}_{\cA^{e}}(\cA, \cA\otimes_{\mathcal{O}_{X}}\mathcal{M})$$  
\end{proposition}
\begin{proof}
Let $\{U_i\}$ be an affine open cover of $X$, where each $U_i = \operatorname{Spec} R_i$. Suppose that $\mathcal{M}|_{U_i} \cong \widetilde{M_i}$ for some $R_i$-module $M_i$. By Lemma \ref{bimodule homorphism}, we have isomorphism of $\mathcal{O}_{U_{i}}$-modules: 
$$\eta_{i}:  \mathcal{M}|_{U_{i}}\xrightarrow{\sim} \mathcal{H}om_{\cA^{e}}(\cA, \cA\otimes_{\mathcal{O}_{X}} \mathcal{M})|_{U_{i}}. $$

By Lemma~\ref{bimodule homorphism} again, the maps $\eta_i$ agree on the intersections $U_i \cap U_j$. Therefore, by gluing $\eta_i$'s together, we obtain a global isomorphism of $\mathcal{O}_{X}$-modules:
\[
\eta : \mathcal{M} \xrightarrow{\sim} \mathcal{H}om_{\mathcal{A}^{e}}(\mathcal{A}, \mathcal{A} \otimes_{\mathcal{O}_{X}} \mathcal{M}).
\]
Thus we have $$\HH^{0}(X,\mathcal{M}) \cong \HH^{0}(X,\mathcal{H}om_{\mathcal{A}^{e}}(\mathcal{A}, \mathcal{A} \otimes_{\mathcal{O}_{X}} \mathcal{M}))= \operatorname{Hom}_{\cA^{e}}(\cA, \cA\otimes_{\mathcal{O}_{X}}\mathcal{M}).$$
By Lemma~\ref{bimodule homorphism}, the  isomorphism  is given by
\[
\eta_X \colon \HH^{0}(X,\mathcal{M}) \xrightarrow{\;\sim\;} 
\operatorname{Hom}_{\mathcal{A}^{e}}\!\bigl(\mathcal{A},\, \mathcal{A}\otimes_{\mathcal{O}_{X}}\mathcal{M}\bigr),
\qquad
m \longmapsto \bigl(a \longmapsto a \otimes m\bigr).
\]
  
\end{proof}
\begin{remark}\label{induced by line bundle}
By  Proposition \ref{Bimodule of sheaf}, we know for any $\psi\in\operatorname{Hom}_{\cA^{e}}(\cA, \cA \otimes_{\mathcal{O}_{X}} \mathcal{M})$, $\psi=\operatorname{id}_{\cA}\otimes \phi, $ where $\phi\in \operatorname{Hom}_{\mathcal{O}_{X}}(\mathcal{O}_{X},\mathcal{M})$.   
\end{remark}
\begin{remark}
The proof of Proposition~\ref{Bimodule of sheaf} cannot, in general, be extended to Deligne--Mumford stacks via descent. For instance, let
$
p \colon Y \longrightarrow [Y/G] = X,
$
where $Y$ is a scheme and $G$ is a finite group. Over $Y$ we obtain an isomorphism
\[
\varphi \colon
p^{*}\mathcal{M} \xrightarrow{\;\sim\;}
p^{*}\mathcal{H}om_{\mathcal{A}^{e}}
\!\bigl(\mathcal{A},\, \mathcal{A} \otimes_{\mathcal{O}_{X}} \mathcal{M}\bigr).
\]
However, it is not clear whether \(\varphi\) is \(G\)-equivariant. Consequently, we cannot conclude that
$
\mathcal{M} \;\cong\;
\mathcal{H}om_{\mathcal{A}^{e}}\!\bigl(\mathcal{A},\, \mathcal{A} \otimes_{\mathcal{O}_{X}} \mathcal{M}\bigr)
$
over $X$.

\end{remark}

    \begin{proposition}\label{trivial line bundle}
    We have the isomorphism $\mathcal{A} \cong \mathcal{A} \otimes_{\mathcal{O}_{X}} \mathcal{M}$ as $\mathcal{A}$-bimodules if and only if $\mathcal{M}$ is isomorphic to $\mathcal{O}_X$ i.e. $\mathcal{M} \cong \mathcal{O}_X$.
\end{proposition}

\begin{proof}
    One direction is clear. For the other direction, suppose $\mathcal{A} \cong \mathcal{A} \otimes_{\mathcal{O}_{X}} \mathcal{M}$ as $\mathcal{A}$-bimodules. Let $\psi: \mathcal{A} \xrightarrow{\sim} \mathcal{A} \otimes_{\mathcal{O}_{X}} \mathcal{M}$ denote the isomorphism. By Remark~\ref{induced by line bundle}, we have $\psi = \operatorname{id}_{\mathcal{A}} \otimes \phi$ for some $\phi \in \operatorname{Hom}_{\mathcal{O}_X}(\mathcal{O}_X, \mathcal{M})$. Note that we have the following exact sequence:
\[
0 \to \operatorname{Ker}(\phi) \to \mathcal{O}_X \xrightarrow{\phi} \mathcal{M} \to \operatorname{Coker}(\phi) \to 0.
\]
After tensoring with $\mathcal{A}$, we obtain the following exact sequence:
\[
0 \to \mathcal{A} \otimes_{\mathcal{O}_{X}} \operatorname{Ker}(\phi) \to \mathcal{A} \xrightarrow{\psi} \mathcal{A} \otimes_{\mathcal{O}_{X}} \mathcal{M} \to \mathcal{A} \otimes_{\mathcal{O}_{X}} \operatorname{Coker}(\phi) \to 0.
\]
Since $\psi$ is an isomorphism, we have $\mathcal{A} \otimes_{\mathcal{O}_{X}} \operatorname{Ker}(\phi) = \mathcal{A} \otimes_{\mathcal{O}_{X}} \operatorname{Coker}(\phi) = 0$.  
Note that $\mathcal{A}$ is locally free of finite rank as an $\mathcal{O}_X$-module, so we have $\operatorname{Ker}(\phi) = \operatorname{Coker}(\phi) = 0$.  
Thus, $\phi$ is an isomorphism, and hence $\mathcal{M} \cong \mathcal{O}_X$. 
\end{proof}
Now we could prove the main theorem in this paper.

\begin{theorem}\label{injective of module}

The functor
\[
- \otimes_{\mathcal{O}_{X}} \mathcal{M} : \operatorname{QCoh}(X, \mathcal{A}) \to \operatorname{QCoh}(X, \mathcal{A})
\]
is naturally isomorphic to the identity functor $\operatorname{id}$ if and only if $\mathcal{M} \cong \mathcal{O}_X$.

\end{theorem}

\begin{proof}
The theorem follows from Propositions~\ref{Proposition of fix category} and~\ref{trivial line bundle}.
\end{proof}

Let $\operatorname{Coh}(X,\mathcal{A})$ denote the abelian category of coherent right $\mathcal{A}$-modules, and set 
$D^{b}(X,\mathcal{A}) := D^{b}(\operatorname{Coh}(X,\mathcal{A}))$ for its bounded derived category. 
Analogous to Theorem~\ref{injective of module}, we obtain the following theorem for $\operatorname{Coh}(X,\mathcal{A})$ when $X$ is a Noetherian scheme.

\begin{theorem}\label{injective of Coh}
Assume that $X$ is Noetherian and that $\mathcal{M} \in \operatorname{Coh}(X)$. Then the functor
\[
- \otimes_{\mathcal{O}_{X}} \mathcal{M} : \operatorname{Coh}(X, \mathcal{A}) \to \operatorname{Coh}(X, \mathcal{A})
\]
is naturally isomorphic to the identity functor $\operatorname{id}$ if and only if $\mathcal{M} \cong \mathcal{O}_X$.
Consequently, the map
\[
\Psi: \operatorname{Pic}(X) \to \operatorname{Aut}\!\big(D^{b}(X, \mathcal{A})\big), 
\qquad \mathcal{L} \mapsto \Psi(\mathcal{L}): \mathcal{F}^{\bullet} \mapsto \mathcal{F}^{\bullet} \otimes_{\mathcal{O}_{X}} \mathcal{L},
\]
is injective.
\end{theorem}

\begin{proof}
By \cite[Proposition~3.1.1.9]{lieblich2008twisted}, the ind-completion of $\operatorname{Coh}(X,\mathcal{A})$ is $\operatorname{QCoh}(X,\mathcal{A})$.
By \cite[Proposition~6.1.9]{kashiwara2006categories}, a natural isomorphism
\[
-\otimes_{\mathcal{O}_{X}} \mathcal{M} \;\cong\; \operatorname{id}
\quad\text{on }\operatorname{Coh}(X,\mathcal{A})
\]
extends to a natural isomorphism
\[
-\otimes_{\mathcal{O}_{X}} \mathcal{M} \;\cong\; \operatorname{id}
\quad\text{on }\operatorname{QCoh}(X,\mathcal{A}).
\]
The claim then follows from Theorem~\ref{injective of module}.
\end{proof}

Now assume that $X$ is a smooth projective variety over a field $k$. Let $\mathscr{X} \to X$ be the $\mathbb{G}_m$-gerbe associated with the Brauer class $\alpha$. Let $\mathscr{A}ut_{\mathscr{X}}$ be the fibered category which assigns to each $k$-scheme $S$ the groupoid of automorphisms $\mathscr{X} \to \mathscr{X}$ over $S$ that induce the identity on the stabilizer group schemes $\mathbb{G}_m$. Then $\mathscr{A}ut_{\mathscr{X}}$ acts on $D^{b}(X, \alpha)$. We prove that the action is faithful.

\begin{corollary}\label{action faithful}
Assume that $k$ is algebraically closed.   That is let $\sigma \in \mathscr{A}ut_{\mathscr{X}}(k)$  be an automorphism of the stack inducing the identity funtor on $D^{b}(X,\alpha)$, then $\sigma=\operatorname{id}.$
\end{corollary}
\begin{proof}
By \cite[Proposition 6.5]{olsson2025twisted}, we know that $\sigma$ is given by tensoring with a line bundle $\mathcal{L}$ on $X$. By Theorem \ref{injective of Coh}, we know $\mathcal{L}\cong \mathcal{O}_{X}$. We complete the proof.
\end{proof}

Now we want to show that $D^{b}(X)$ is Calabi--Yau if and only if $D^{b}(X, \alpha)$ is Calabi--Yau. To do so, we first need the following lemma.

\begin{lemma}[{Serre duality}]\label{Serre duality}
Let $n=\operatorname{dim}(X)$. Then $\omega_{\mathcal{A}}[n]$ is the dualizing object for $D^{b}(X, \mathcal{A})$, where the $\mathcal{A}$-bimodule $\omega_{\mathcal{A}}$ is defined as $\omega_{\mathcal{A}} := \mathcal{A} \otimes_{\mathcal{O}_{X}} \omega_X$. Moreover, we have the following version of Serre duality:
\[
\operatorname{Hom}_{D^{b}(X, \mathcal{A})}(\mathcal{F}^\bullet, \mathcal{G}^\bullet)
\cong
\operatorname{Hom}_{D^{b}(X, \mathcal{A})}(\mathcal{G}^\bullet, \mathcal{F}^\bullet \otimes_{\cA} \omega_{\mathcal{A}}[n])^{*}
\cong 
\operatorname{Hom}_{D^{b}(X, \mathcal{A})}(\mathcal{G}^\bullet, \mathcal{F}^\bullet \otimes_{\mathcal{O}_{X}} \omega_{X}[n])^{*}
\]
for any $\mathcal{F}^\bullet, \mathcal{G}^\bullet \in D^{b}(X, \mathcal{A})$.
\end{lemma}
\begin{proof}
See, for example, \cite[Example~6.4]{yekutieli2006dualizing}.
\end{proof}
\begin{theorem}\label{Calabi-Yau}
Let $n=\dim(X)$. Then
    $D^{b}(X)$ is a Calabi--Yau $[n]$-category if and if $D^{b}(X,\alpha)$ is a Calabi--Yau $[n]$-category.
\end{theorem}
\begin{proof}
    Recall that $\cA$ is the Azumaya algebra on $X$ such that $[\cA]=\alpha.$

$\Longrightarrow$: If $D^{b}(X)$ is a Calabi--Yau $[n]$-category, then $\omega_{X}\cong \mathcal{O}_{X}$. Then the result follows from Lemma \ref{Serre duality}.

$\Longleftarrow$: If $D^{b}(X,\alpha)=D^{b}(X,\cA)$ is a Calabi--Yau $[n]$-category. Then for any $\mathcal{F}^{\bullet}, \mathcal{G}^{\bullet}\in D^{b}(X,\cA)$  we have $$\operatorname{Hom}_{D^{b}(X, \mathcal{A})}(\mathcal{F}^\bullet, \mathcal{G}^\bullet)
\cong
\operatorname{Hom}_{D^{b}(X, \mathcal{A})}(\mathcal{G}^\bullet, \mathcal{F}^\bullet[n])^{*}.
$$
By Lemma \ref{Serre duality}, we have 
$$
\operatorname{Hom}_{D^{b}(X, \mathcal{A})}(\mathcal{G}^\bullet, \mathcal{F}^\bullet[n])\cong 
\operatorname{Hom}_{D^{b}(X, \mathcal{A})}(\mathcal{G}^\bullet, \mathcal{F}^\bullet \otimes_{\mathcal{O}_{X}} \omega_{X}[n]).
$$
So $\mathcal{F}^{\bullet}\cong \mathcal{F}^{\bullet}\otimes_{\mathcal{O}_{X}} \omega_{X}$ for all $\mathcal{F}^{\bullet}$. Thus by Theorem \ref{injective of Coh}, $\omega_{X}\cong \mathcal{O}_{X}$. Thus, $D^{b}(X)$ is a Calabi--Yau $[n]$-category. 
    
\end{proof}

In \cite[Example~6.6]{olsson2025twisted}, Olsson gave an example showing that Theorem~\ref{injective of Coh} does not hold for Deligne--Mumford stacks. We provide a new example demonstrating that Theorems \ref{injective of module}, \ref{injective of Coh}, and \ref{Calabi-Yau} may fail in the setting of Deligne--Mumford stacks.

\begin{example}\label{example} 
Let $X$ be an Enriques surface over $\mathbb{C}$. By \cite[Chapter~VIII]{barth2004compact}, we have
$$
\mathrm{H}^{1}_{\acute{e}t}(X, \mu_2) \cong \operatorname{Pic}(X)[2] \cong \mathbb{Z}/2\mathbb{Z},
$$
which classifies 2-torsion line bundles on $X$. The nontrivial class corresponds to the canonical bundle $\omega_X$, satisfying $\omega_X^{\otimes 2} \cong \mathcal{O}_X$. This  $\mu_{2}$-torsor determines an étale double cover
$$
p: Y \to X,
$$
where $Y$ is a K3 surface. By \cite{beauville2009brauer}, we know that $\Br(X) = \mathbb{Z}/2\mathbb{Z}$. Let
\[
\pi: \mathbf{B}\mu_{2,X} = [X/\mu_{2}] \to X
\]
be the trivial $\mu_{2}$-gerbe over $X$. By \cite[Lemma~3.0.1]{yu2025morita}, we have
\[
\Br(\mathbf{B}\mu_{2, X}) = \Br(X) \oplus \HH^{1}_{\acute{e}t}(X, \mu_2) = \mathbb{Z}/2\mathbb{Z} \oplus \mathbb{Z}/2\mathbb{Z}.
\]
Note that by \cite[Theorem~1.5]{ishii2015special}, we have
\[
D^{b}(\mathbf{B}\mu_{2, X}) = D^{b}(X) \oplus D^{b}(X),
\]
which is not a Calabi--Yau category.

Let $\alpha \in \HH_{\acute{e}t}^{1}(X, \mu_2) \subseteq \Br(\mathbf{B}\mu_{2, X})$ be the Brauer class of the $\mathbf{B}\mu_{2, X}$ corresponding to the double cover $p: Y \to X$. By \cite[Lemma~5.1.7]{yu2025morita}, we have
\[
D^{b}(\mathbf{B}\mu_{2, X}, \alpha) \cong D^{b}(Y),
\]
which is a Calabi--Yau $[2]$-category. This shows that Theorem~\ref{Calabi-Yau} does not hold for Deligne--Mumford stacks.

At the same time, since $D^{b}(\mathbf{B}\mu_{2, X}, \alpha)$ is a Calabi--Yau $[2]$-category, the same argument as in the proof of Theorem~\ref{Calabi-Yau} shows that
\[
\mathcal{F}^\bullet \otimes_{\mathcal{O}_{\mathbf{B}\mu_{2, X}}} \omega_{\mathbf{B}\mu_{2, X}} \cong \mathcal{F}^\bullet
\]
for all $\mathcal{F}^\bullet \in D^{b}(\mathbf{B}\mu_{2, X}, \alpha)$, where $\omega_{\mathbf{B}\mu_{2, X}}$ is the canonical bundle. Note that $\omega_{\mathbf{B}\mu_{2, X}} = p^* \omega_X$, which is not trivial. This shows that Theorems~\ref{injective of module} and \ref{injective of Coh}  does not hold for Deligne--Mumford stacks.

\end{example}

Note that in Example~\ref{example}, every (geometric) point of the Deligne--Mumford stack $\mathbf{B}\mu_{2,X}$ has stabilizer group $\mu_{2}$, which is nontrivial. 
Now assume that $X$ is a quasi-separated Deligne--Mumford stack with generically trivial stabilizer. 
Then there exists a dense open substack $V \hookrightarrow X$ such that $V$ is an algebraic space. 
By \cite[Theorem~4.5.1]{alper2024stacks}, there exists a dense open subspace $U \hookrightarrow V$ such that $U$ is a scheme.

\begin{proposition}\label{proposition de}
Let $X$ be a quasi-separated Deligne--Mumford stack with generically trivial stabilizer, and let $U \hookrightarrow X$ be a dense open substack which is a scheme. 
Assume $\mathcal{M} \in \operatorname{QCoh}(X)$ and $\alpha \in \Br(X)$. 
If the functor
\[
- \otimes_{\mathcal{O}_{X}} \mathcal{M} : \operatorname{QCoh}(X, \alpha) \to \operatorname{QCoh}(X, \alpha)
\]
is naturally isomorphic to the identity functor $\operatorname{id}$, then $\mathcal{M}|_{U} \cong \mathcal{O}_{U}$.
\end{proposition}

\begin{proof}
Let $i \colon U \hookrightarrow X$ be the open immersion. 
For any $\mathcal{F} \in \operatorname{QCoh}(U, i^{*}\alpha)$, we have $i_{*}\mathcal{F} \in \operatorname{QCoh}(X,\alpha)$ and, by assumption,
\[
i_{*}\mathcal{F} \;\cong\; i_{*}\mathcal{F} \otimes_{\mathcal{O}_{X}} \mathcal{M}.
\]
Since $i^{*}i_{*} \cong \operatorname{id}$ , we obtain
\[
\mathcal{F} \;\cong\; i^{*}\!\bigl(i_{*}\mathcal{F} \otimes_{\mathcal{O}_{X}} \mathcal{M}\bigr)
\;\cong\; \mathcal{F} \otimes_{\mathcal{O}_{U}} i^{*}\mathcal{M}.
\]
Let $f \colon \mathcal{F} \to \mathcal{G}$ be a morphism in $\operatorname{QCoh}(U, i^{*}\alpha)$. We obtain the commutative diagram
\[
\begin{tikzcd}
i_{*}\mathcal{F} \arrow[r,"i_{*}f"] \arrow[d,"\cong"'] &
i_{*}\mathcal{G} \arrow[d,"\cong"] \\
i_{*}\mathcal{F}\otimes_{\mathcal{O}_{X}}\mathcal{M} \arrow[r,"i_{*}f \,\otimes\, \operatorname{id}"] &
i_{*}\mathcal{G}\otimes_{\mathcal{O}_{X}}\mathcal{M}.
\end{tikzcd}
\]
Applying the functor $i^{*}$  yields the commutative diagram
\[
\begin{tikzcd}
\mathcal{F} \arrow[r,"f"] \arrow[d,"\cong"] &
\mathcal{G} \arrow[d,"\cong"] \\
\mathcal{F}\otimes_{\mathcal{O}_{U}} i^{*}\mathcal{M} \arrow[r,"f \,\otimes\, \operatorname{id}"] &
\mathcal{G}\otimes_{\mathcal{O}_{U}} i^{*}\mathcal{M}.
\end{tikzcd}
\]
Therefore, the functor
\[
- \otimes_{\mathcal{O}_{U}} i^{*}\mathcal{M} \colon \operatorname{QCoh}(U, i^{*}\alpha) \longrightarrow \operatorname{QCoh}(U, i^{*}\alpha)
\]
is naturally isomorphic to the identity. Since $U$ is a scheme, Theorem~\ref{injective of module} implies that $i^{*}\mathcal{M} \cong \mathcal{O}_{U}$.
\end{proof}
We do not know if Theorem \ref{injective of module} holds for algebraic stacks:


\begin{conjecture}\label{conjecture}
Let $X$ be a quasi-separated Deligne--Mumford stack with generically trivial stabilizer. 
Let $\mathcal{M}\in \operatorname{QCoh}(X)$ and $\alpha\in \Br(X)$.
Then the functor
\[
-\otimes_{\mathcal{O}_{X}}\mathcal{M} \colon \operatorname{QCoh}(X,\alpha)\longrightarrow \operatorname{QCoh}(X,\alpha)
\]
is naturally isomorphic to the identity functor $\operatorname{id}$ if and only if $\mathcal{M}\cong \mathcal{O}_{X}$.
\end{conjecture}

\section*{References}
\bibliographystyle{alpha}
\renewcommand{\section}[2]{} 
\bibliography{ref}

\end{document}